\documentclass[a4paper,reqno]{amsart}
\usepackage[foot]{amsaddr}
\usepackage[utf8]{inputenc}
\usepackage[T1]{fontenc}
\usepackage[american]{babel}
\usepackage{amsmath,amssymb,amsthm}
\usepackage{xcolor}
\usepackage[all,cmtip]{xy}
\usepackage{mathtools}
\usepackage{hyperref}
\usepackage{stmaryrd}

\newcommand{\verticalbinomial}{\genfrac{|}{|}{0pt}{}}
\newcommand{\LieAlg}{\mathbf{Lie}_R}

\newcommand{\Grp}{\mathbf{Grp}}

\newcommand{\XMod}{\mathbf{XMod}}
\newcommand{\Set}{\mathbf{Set}}

\newcommand{\Act}{\mathbf{Act}}
\newcommand{\Mag}{\mathit{Mag}}
\newcommand{\Ker}{\mathit{Ker}}
\newcommand{\Coker}{\mathit{Coker}}
\newcommand{\A}{\mathbb{A}}

\theoremstyle{plain} 
\newtheorem{thm}{Theorem}[section] 
 
\newtheorem{lemma}[thm]{Lemma} 
\newtheorem{prop}[thm]{Proposition}

\theoremstyle{definition} 
\newtheorem{defi}[thm]{Definition}

\theoremstyle{remark}
\newtheorem{rmk}[thm]{Remark}
\newtheorem{ex}[thm]{Example}

\title{Compatible Actions of Lie Algebras}
\author{Davide di Micco}
\date{\today}

\address{Università degli Studi di Milano, Via Saldini 50, 20133 Milano, Italy}
\email{davide.dimicco@unimi.it}

\begin{document}
\begin{abstract}
We study compatible actions (introduced by Brown and Loday in their work on the non-abelian tensor product of groups) in the category of Lie algebras over a fixed ring. We describe the Peiffer product via a new diagrammatic approach, which specializes to the known definitions both in the case of groups and in the case of Lie algebras. We then use this approach to transfer a result linking compatible actions and pairs of crossed modules over a common base object $L$ from groups to Lie algebras. Finally, we show that the Peiffer product, naturally endowed with a crossed module structure, has the universal property of the coproduct in~$\XMod_L(\LieAlg)$.
\end{abstract}


\maketitle

\section*{Introduction}

The aim of this paper is to study compatible actions of Lie algebras (introduced in~\cite{Ell91}) and to compare them with compatible actions of groups (first studied in~\cite{BL84}). With this idea in mind, we try to use a diagrammatic and internal approach whenever it is possible: to do so we take advantage of the equivalence between the internal actions (introduced in~\cite{Jan03}) and the usual actions of Lie algebras, as well as the equivalence between internal crossed modules and crossed modules of Lie algebras (see~\cite{Jan03}).

In Brown and Loday's article~\cite{BL84} it is stated that two groups $M$ and $N$ act on each other compatibly if and only if there exists a group $L$ and two crossed module structures $(M\xrightarrow{\mu}L,\psi_M)$ and $(N\xrightarrow{\nu}L,\psi_N)$.

One of the two implications above in the Lie algebra case has been mentioned by Ellis in~\cite{Ell91} while the other appears as a remark in~\cite{Khm99}. We provide a proof of this result which, thanks to its intrinsic form, is valid in both cases. In order to do so we need to consider the \emph{Peiffer product} of two Lie algebras acting on each other compatibly (corresponding to the Peiffer product of groups, so named in~\cite{GH87}, but first defined in~\cite{Whi41}): this is already present in~\cite{Khm99}, but we use a different (yet equivalent) construction which is the same for groups and Lie algebras.

A consequence of this result is that the non-abelian tensor product of Lie algebras introduced in~\cite{Ell91} can naturally be interpreted as a tensor product of compatible actions or as a tensor product of crossed modules over a common base object.

Finally we prove in Theorem~\ref{thm:remark 2.16 in BL in LieAlg} that the Peiffer product can be endowed with a crossed module structure making it a coproduct in $\XMod_L(\LieAlg)$ exactly as proved in~\cite{Bro84} in the case of groups.

As a consequence we get that the internal definition of the Peiffer product given in~\cite{CMM17} coincides with the one introduced in~\cite{Khm99}.

The paper is organized as follows. In the first section we recall basic definitions and results. In the second section we show the link between the notions of compatible actions for groups and for Lie algebras, giving the idea of a possible generalization to a semi-abelian category~\cite{JMT02} as it will be considered in the paper~\cite{dMVdL19} in preparation; we show that two crossed modules with a common codomain in $\LieAlg$ induce compatible actions and, in order to prove the converse, we first give an internal construction of the Peiffer product of two Lie algebras and then we endow it with crossed module structures. Lastly, in the third section we show that the coproduct in $\XMod_L(\LieAlg)$ can be obtained through the Peiffer product and we draw some consequences of this result.

\section{Preliminaries}
We start by recalling some well-known facts that we will need in the following and in the meantime we use this section to fix some notation.

\begin{defi}
\label{defi:R-Lie algebra}
Let $R$ be a commutative ring and let $M$ be an $R$-module. We say that $M$ is a \textit{Lie algebra over} $R$ if it is endowed with a binary operation 
\[
[-,-]\colon M\times M\to M
\]
called \textit{Lie bracket}, such that the following conditions hold:
\begin{itemize}
\item[1)] $[ax+by,z]=a[x,z]+b[y,z]$ and $[x,ay+bz]=a[x,y]+b[x,z]$ ($R$-bilinearity);
\item[2)] $[x,x]=0$ and $[x,y]+[y,x]=0$ (alternating);
\item[3)] $[[x,y],z]+[[y,z],x]+[[z,x],y]=0$ (Jacobi identity).
\end{itemize}
\end{defi}

\begin{rmk}
We recall that the above definition is redundant: notice that the two conditions in $1)$ are equivalent due to $2)$, so it suffices to check just one of them. Moreover, $[x,x]=0$ always implies $[x,y]+[y,x]=0$, and the converse is true whenever the multiplication by $2$ is injective in $M$ (that is, $M$ is $2$-torsion free). Furthermore, the equation $[[x,y],z]+[[y,z],x]+[[z,x],y]=0$ is equivalent to $[x,[y,z]]+[y,[z,x]]+[z,[x,y]]=0$ thanks to $2)$.
\end{rmk}

\begin{defi}
Let $M$ and $L$ be $R$-Lie algebras. A morphism of $R$-Lie algebras $f\colon M\to L$ is a morphism of $R$-modules such that 
\[
f([x,y])=[f(x),f(y)].
\]
This defines the category $\LieAlg$ of $R$-Lie algebras and $R$-Lie algebra morphisms.
\end{defi}

\begin{rmk}
There is an obvious forgetful functor $U\colon \LieAlg\to \Set$ and it has a left adjoint $F\colon\Set\to \LieAlg$: this functor builds the free $R$-Lie algebra on a given set $X$ with the following well-known procedure.
\begin{itemize}
 \item[i)] First of all we build the free magma on $X$, denoted $\Mag(X)$, writing $[-,-]\colon \Mag(X)\times \Mag(X)\to \Mag(X)$ for the binary operation: this means that an element of $\Mag(X)$ is given by a word with square brackets, as for instance \lq\lq$[[x_1,[x_2,x_3]],x_4]$\rq\rq.
 \item[ii)] Then we take the free $R$-module on it $R[\Mag(X)]$ and we extend the product by defining 
 \[
 \left[\sum\limits_{i=0}^nr_ix_i,\sum\limits_{j=0}^ms_jy_j\right]=\sum\limits_{i=0}^n\sum\limits_{j=0}^mr_is_j[x_i,y_j].
 \]
 This product gives to $R[\Mag(X)]$ the structure of a $R$-algebra.
 \item[iii)] Finally consider the ideal $I$ generated by the symbols
 \begin{itemize}
 \item[$\bullet$] $[x,x]$,
 \item[$\bullet$] $[x,y]+[y,x]$,
 \item[$\bullet$] $[x,[y,z]]+[y,[z,x]]+[z,[x,y]]$,
 \end{itemize}
 with $x,y,z\in X$ and define $F(X):=R[\Mag(X)]/I$.
\end{itemize}
\end{rmk}

\begin{rmk}
Let $M$ and $N$ be two $R$-Lie algebras. Their coproduct $M+N$ is the $R$-Lie algebra given by $F(U(M)\sqcup U(N))/J$ where $J$ is the ideal generated by the identities coming separately from $M$ and from $N$: this means that it is a quotient of the free algebra on the disjoint union of the underlying sets of the two algebras.
\end{rmk}

\begin{defi}
Given a word $s\in M+N$, we say that it is \emph{well nested} if it is a simple bracket---$[x_1,x_2]$ where $x_1$, $x_2\in M\cup N$---or if it is obtained by taking the bracket of an element with a well-nested word. Equivalently this means that $s$ does not contain a bracket between two brackets. The \emph{height} of a well nested word is simply the number of pair of brackets appearing in it. Given a word $s\in M+N$, any simple bracket $[x_1,x_2]$ is contained in a maximal well nested subword of $s$ and we say that the \emph{relative height} of $x_1$ and of $x_2$ in $s$ is the height of this subword. 
\end{defi}

Since we couldn't find a clear reference for the following lemma, we prove it here, even if we think it is a well-known result.
\begin{lemma}
\label{lemma:rewriting lemma}
Every element in $M+N$ can be written as a linear combination of elements of the form 
\begin{equation}
\label{eq:good form for lie words}
[x_k,[x_{k-1},[\ldots,[x_3,[x_2,x_1]]\cdots]]]
\end{equation}
with $x_i\in M$ or $x_i\in N$.
\end{lemma}
\begin{proof}
Consider a word $s$ which has $n$ pairs of brackets and apply the following algorithm:
\begin{itemize}
\item[1)] Choose a subword $t$ of $s$ which is well nested: this always exists, because we can take one of the innermost (and hence simple) brackets.
\item[2)] If $t=s$ go to $3)$. Otherwise $t$ is contained in a subword of the form
\[
[t,[w_1,w_2]]
\qquad\text{or}\qquad
[[w_1,w_2],t].
\]
with $w_1$ and $w_2$ subwords of $s$. Use the Jacobi identity to break $[t,[w_1,w_2]]$ into $[w_1,[t,w_2]]+[[t,w_1],w_2]$ (and similarly in the other case). Now $s$ can be seen as the sum of the two words in which we substituted $[t,[w_1,w_2]]$ with the two summands resulted from the application of the Jacobi identity. For each of these words repeat the step $2)$ choosing them as new $s$ and the maximal well nested word containing the old $t$ as new $t$.
\item[3)] Since $s$ is now well nested it suffices to apply the alternating property until all the brackets have a simple element on the left. This has only the effect of possibly changing the sign in front of the word. 
\end{itemize}
The reason why this algorithm works is simply because at each application of $2)$ we obtain one of the following:
\begin{itemize}
 \item[i)] the relative height of $t$ increases by at least $1$: this will eventually lead to the relative height reaching $n$, which means that the word in question is well nested;
 \item[ii)] the complexity of the bracket near $t$ decreases: in one application it goes from $[w_1,w_2]$ to both $w_1$ and $w_2$ which individually contains less brackets than $[w_1,w_2]$. This will eventually lead to $w_1$ or $w_2$ being a single element and hence to $i)$ at the next iteration.\qedhere
\end{itemize}
\end{proof}

\begin{rmk}
\label{rmk:choosing last element is possible}
Notice that for each word $s\in M+N$ and for each letter $x$ in it, we can decompose $s$ as a linear combination of words of the form (\ref{eq:good form for lie words}) in such a way that each word in the decomposition has $x_1=x$. This is possible because, by using the Jacobi identity, we can first decompose $s$ as a linear combination of words in which $x$ appears in a simple bracket. Then we can use the algorithm described in Lemma~\ref{lemma:rewriting lemma} choosing as starting $t$ the simple bracket containing $x$.
\end{rmk}

\begin{defi}
Let $P$ and $M$ be two $R$-Lie algebras. The object $P\flat M$ is defined in~\cite{BJK05,Jan03} as the kernel of the morphism 
\[
P+M\xrightarrow{\binom{1}{0}}P
\]
and it is the key ingredient for the definition of internal actions as we will see in the next section.
An element of this $R$-Lie algebra is an element of $P+M$ such that each of its monomials contains an element from $M$: indeed the arrow $\binom{1}{0}$ takes a linear combination of \lq\lq words\rq\rq\ and sends it to the linear combination of \lq\lq words\rq\rq\ obtained by substituting every element from $M$ with $0$ (therefore only monomials with an element in $M$ go to zero).

Notice that $(P\flat M, k_{P,M})=\Ker(\Coker(i_M\colon M\to P+M))$ and therefore $P\flat M$ is the ideal generated by $M$ in $P+M$.
\end{defi}

\begin{rmk}
Recall from~\cite{Jan03,BJK05} that for each object $P$, the functor $P\flat(-)$ is part of a monad structure. In particular $\eta^P\colon 1_{\LieAlg}\to P\flat(-)$ is given by 
\begin{align*}
 \eta^P_M\colon & M\to P\flat M\colon m\mapsto m
\end{align*}
and $\mu^P\colon P\flat(P\flat(-))\to P\flat(-)$ has components
\begin{align*}
 \mu^P_M\colon P\flat(P\flat M)&\to P\flat M
\end{align*}
which maps the two different brackets in $P\flat(P\flat M)$ to the one bracket in $P\flat M$.

Furthermore if $f\colon A \to B$ is a morphism, then $P\flat(f)=1_P\flat f\colon P\flat A\to P\flat B$ is given by sending each linear combination of words in $P\flat A$ into the one obtained by substituting every element $a\in A$ with its image $f(a)\in B$. 
\end{rmk}

\section{Actions and compatible actions of Lie algebras}
\label{section:Actions and compatible actions of Lie algebras}

We start by recalling the equivalent definitions of action and internal action in $\LieAlg$.

\begin{defi}
\label{defi:particular definition of actions in liealg}
Let $M$ and $P$ be $R$-Lie algebras. An \emph{action of $P$ on $M$} is given by a $R$-bilinear map $\psi\colon P\times M\to M$ with $(p,m) \mapsto {}^pm=\psi(p,m)$, such that for each $p,p'\in P$ and $m,m'\in M$ we have
\begin{itemize}
 \item ${}^{[p,p']}m={}^{p}({}^{p'}m)-{}^{p'}({}^{p}m)$ and
 \item ${}^{p}[m,m']=[{}^{p}m,m']+[m,{}^{p}m']$
\end{itemize}
\end{defi}

In~\cite{BJK05} Borceux, Janelidze and Kelly introduced the definition of \emph{internal action} in the context of semi-abelian categories and they proved that it is a generalization of the different particular definitions such as the one that we just stated for Lie algebras.

\begin{defi}
An \emph{internal action} $\xi\colon P\flat M\to M$ is an algebra for the monad $(P\flat -, \eta^P,\mu^P)$ for some $P$. This means that it is a morphism of $R$-Lie algebras such that the diagrams
\begin{align*}
\xymatrix{
M \ar[r]^-{\eta^P_M} \ar@{=}[rd] & P\flat M \ar[d]^-{\xi}\\
& M
}
&&
\xymatrix{
P\flat(P\flat M) \ar[r]^-{\mu^P_M} \ar[d]_-{1_P\flat\xi} & P\flat M \ar[d]^-{\xi}\\
P\flat M \ar[r]_-{\xi} & M
}
\end{align*}
commute. That is, such that
\begin{align*}
\xi(m)=m, && \xi(\mu^P_M(s))=\xi((1_P\flat\xi)(s))
\end{align*}
for all $m\in M$ and $s\in P\flat(P\flat M)$. For example if $s=\{p,[m,p']\}$, then $\mu^P_M(s)=[p,[m,p']]$ and $(1_P\flat\xi)(s)=[p,\xi([m,p'])]$, so we want that 
\[
\xi([p,[m,p']])=\xi([p,\xi([m,p'])]).
\]
This means that the image of the action on a complicated word can be obtained by taking the image of the most internal bracket and iterating this process until there are no brackets left. We will call this property \emph{decomposability}.
The actions just defined form a category, denoted by $\Act(\LieAlg)$.
\end{defi}

\begin{rmk}
It is easy to notice that there is an equivalence between actions and internal actions. In particular this correspondence sends an internal action $\xi\colon P\flat M\to M$ to the action $\psi\colon P\times M\to M$ defined via $\psi(p,m):=\xi([p,m])$, and conversely it sends an action $\psi\colon P\times M\to M$ to the internal action $\xi\colon P\flat M\to M$ defined via
\[
\begin{cases}
\xi(m):=m,\\
\xi([p,m]):=\psi(p,m).
\end{cases}
\]
The behavior of $\xi$ on more complex elements is uniquely determined by the hypothesis of decomposability.
From now one we are going to use actions or internal actions equivalently, depending on which is the more convenient approach in each specific case.
\end{rmk}

\begin{ex}
Given an $R$-Lie algebra $M$ we always have an action of $M$ on itself, that is the \textit{conjugation action} $M\times M\to M$ given by $(m,m')\mapsto [m,m']$. Viewed as an internal action, it is $\chi_M\colon M\flat M\to M$.
\end{ex}

\begin{defi}
\label{defi:definition of the coproduct action in LieAlg}
Consider an action $\xi^N_M\colon N\flat M\to M$ and the conjugation\linebreak $\chi_M\colon M\flat M\to M$. We can always construct an action $\xi^{M+N}_M\colon (M+N)\flat M\to M$ of the coproduct $M+N$ on $M$ such that it extends both $\xi^N_M$ and $\chi_M$. It is defined via
\begin{itemize}
 \item $\overline{m}\longmapsto \overline{m}$,
 \item $[m,\overline{m}]\longmapsto [m,\overline{m}]$,
 \item $[n,\overline{m}]\longmapsto \xi^N_M([n,\overline{m}])$.
\end{itemize}
where $\overline{m}\in M$ and $m,n\in M+N$. Notice that the images of those three types of elements are fixed by the fact that $\xi^{M+N}_M$ is an action and by the fact that it extends both the conjugation of $M$ and the action $\xi^N_M$. Furthermore it is uniquely determined by these requirements since we can easily deduce the behavior of $\xi^{M+N}_M$ on more complex elements by using the Jacobi identity and the decomposability of the action $\xi^{M+N}_M$. For example we can show that
\[
\xi^{M+N}_M([[n,m],\overline{m}])=[\xi^N_M([n,m]),\overline{m}]
\]
by the following chain of equalities
\begin{align*}
\xi^{M+N}_M([[n,m],\overline{m}])&=\xi^{M+N}_M(-[[m,\overline{m}],n]-[[\overline{m},n],m])\\
&=\xi^{M+N}_M([n,[m,\overline{m}]]-[m,[n,\overline{m}]])\\
&=\xi^{M+N}_M([n,[m,\overline{m}]])-\xi^{M+N}_M([m,[n,\overline{m}]])\\
&=\xi^{M+N}_M([n,\xi^{M+N}_M([m,\overline{m}])])-\xi^{M+N}_M([m,\xi^{M+N}_M([n,\overline{m}])])\\
&=\xi^N_M([n,\chi_M([m,\overline{m}])])-\chi_M([m,\xi^N_M([n,\overline{m}])])\\
&=\xi^N_M([n,[m,\overline{m}]])-[m,\xi^N_M([n,\overline{m}])]\\
&=[\xi^N_M([n,m]),\overline{m}]+[m,\xi^N_M([n,\overline{m}])]-[m,\xi^N_M([n,\overline{m}])]\\
&=[\xi^N_M([n,m]),\overline{m}]
\end{align*}
\end{defi}

\begin{defi}
\label{defi:compatible external actions in LieAlg}
Given two $R$-Lie algebras $M$ and $N$, we say that two actions
\begin{align*}
\psi^M_N\colon M\times N\to N && \psi^N_M\colon N\times M\to M
\end{align*}
are \emph{compatible} (see~\cite{Ell91}) if the following equations hold
\begin{equation}
\label{eq:explicit compatibility condition for LieAlg}
\begin{cases}
{}^{({}^{n}m)}n'=\left[n',{}^{m}n\right],\\ {}^{({}^{m}n)}m'=\left[m',{}^{n}m\right].
\end{cases}
\end{equation}
\end{defi}

\begin{rmk}
\label{rmk:explanation of the definition of compatible actions in LieAlg}
The link between this definition and the compatibility condition in the case of groups is given by the following general idea: \emph{the element ${}^{m}n$ (resp. ${}^{n}m$) has to act as the formal conjugation of $m$ and $n$ in the coproduct would do}. In particular in $\Grp$ this amounts to require the equalities
\begin{equation}
\label{eq:compatibility conditions in grp}
\begin{cases}
{}^{\left({}^{n}m\right)}n'={}^{\left(nmn^{-1}\right)}n',\\
{}^{\left({}^{m}n\right)}m'={}^{\left(mnm^{-1}\right)}m',
\end{cases}
\end{equation}
(see~\cite{BL84} for further details) whose internal translation is given by
\begin{align*}
\begin{cases}
\xi^M_N\left(\xi^N_M(x)n'{\xi^N_M(x)}^{-1}\right)=\xi^{M+N}_N\left(xn'x^{-1}\right),\\
\xi^N_M\left(\xi^M_N(y)m'{\xi^M_N(y)}^{-1}\right)=\xi^{M+N}_M\left(ym'y^{-1}\right),
\end{cases}
\end{align*}
with $x=nmn^{-1}$ and $y=mnm^{-1}$. Notice that these can also be seen as the commutativity of the diagrams
\begin{align}
\label{diag:diagrammatic version of compatibility conditions}
\xymatrixcolsep{3pc}
\vcenter{\xymatrix{
(N\flat M)\flat N \ar[r]^-{k_{N,M}\flat 1_N} \ar[d]_-{\xi^N_M\flat 1_N} & (M+N)\flat N \ar[d]^-{\xi^{M+N}_N}\\
M\flat N \ar[r]_-{\xi^M_N} & N
}}
&&
\xymatrixcolsep{3pc}
\vcenter{\xymatrix{
(M\flat N)\flat M \ar[r]^-{k_{M,N}\flat 1_M} \ar[d]_-{\xi^M_N\flat 1_M} & (M+N)\flat M \ar[d]^-{\xi^{M+N}_M}\\
N\flat M \ar[r]_-{\xi^N_M} & M
}}
\end{align}
Besides (\ref{eq:compatibility conditions in grp}), we should also require the equalities
\[
\begin{cases}
{}^{\left({}^{n}m\right)}m'={}^{\left(nmn^{-1}\right)}m',\\
{}^{\left({}^{m}n\right)}n'={}^{\left(mnm^{-1}\right)}n',
\end{cases}
\]
or their internal version
\[
\begin{cases}
\chi_M\left(\xi^N_M(x)m'{\xi^N_M(x)}^{-1}\right)=\xi^{M+N}_M\left(xn'x^{-1}\right),\\
\chi_N\left(\xi^M_N(y)n'{\xi^M_N(y)}^{-1}\right)=\xi^{M+N}_N\left(yn'y^{-1}\right),
\end{cases}
\]
coming from the commutativity of the diagrams
\begin{align*}
\xymatrixcolsep{3pc}
\xymatrix{
(N\flat M)\flat M \ar[r]^-{k_{N,M}\flat 1_M} \ar[d]_-{\xi^N_M\flat 1_M} & (M+N)\flat M \ar[d]^-{\xi^{M+N}_M}\\
M\flat M \ar[r]_-{\chi_M} & M
}
&&
\xymatrixcolsep{3pc}
\xymatrix{
(M\flat N)\flat N \ar[r]^-{k_{M,N}\flat 1_N} \ar[d]_-{\xi^M_N\flat 1_N} & (M+N)\flat N \ar[d]^-{\xi^{M+N}_N}\\
N\flat N \ar[r]_-{\chi_N} & N
}
\end{align*}
However, as one can easily check, these always hold for every pair of actions. 

The same idea applied in $\LieAlg$ leads to the equations
\[
\begin{cases}
{}^{\left({}^{n}m\right)}n'={}^{\left[n,m\right]}n',\\
{}^{\left({}^{m}n\right)}m'={}^{\left[m,n\right]}m',
\end{cases}
\]
whose internal version is given by the system
\[
\begin{cases}
\xi^M_N\left(\left[\xi^N_M\left(\left[n,m\right]\right),n'\right]\right)=\xi^{M+N}_N\left(\left[\left[n,m\right],n'\right]\right),\\
\xi^N_M\left(\left[\xi^M_N\left(\left[m,n\right]\right),m'\right]\right)=\xi^{M+N}_M\left(\left[\left[m,n\right],m'\right]\right),
\end{cases}
\]
or again by the commutativity of (\ref{diag:diagrammatic version of compatibility conditions}).
By using the decomposability of the coproduct actions one can show that these requirements are the same as (\ref{eq:explicit compatibility condition for LieAlg}) in Definition~\ref{defi:compatible external actions in LieAlg}: indeed we have the chains of equalities
\begin{align*}
{}^{\left[n,m\right]}n'&=\xi^{M+N}_N\left(\left[\left[n,m\right],n'\right]\right)=\left[\xi^M_N(\left[n,m\right]),n'\right]=\left[n',\xi^M_N(\left[m,n\right])\right]=\left[n',{}^{m}n\right],\\
{}^{\left[m,n\right]}m'&=\xi^{M+N}_M\left(\left[\left[m,n\right],m'\right]\right)=\left[\xi^N_M(\left[m,n\right]),m'\right]=\left[m',\xi^N_M(\left[n,m\right])\right]=\left[m',{}^{n}m\right].
\end{align*}
Furthermore, in the case of $\LieAlg$ the other two equations 
\[
\begin{cases}
{}^{\left({}^{n}m\right)}m'={}^{\left[n,m\right]}m',\\
{}^{\left({}^{m}n\right)}n'={}^{\left[m,n\right]}n',
\end{cases}
\]
are automatically satisfied: indeed by looking at their internal version
\[
\begin{cases}
\left[\xi^N_M\left(\left[n,m\right]\right),m'\right]=\xi^{M+N}_M\left(\left[\left[n,m\right],m'\right]\right),\\
\left[\xi^M_N\left(\left[m,n\right]\right),n'\right]=\xi^{M+N}_N\left(\left[\left[m,n\right],n'\right]\right),
\end{cases}
\]
one can see that they are precisely a consequence of the decomposability of the coproduct actions shown in Definition~\ref{defi:definition of the coproduct action in LieAlg}.
\end{rmk}

\begin{defi}
\label{defi:definition of crossed module of LieAlg}
A \emph{crossed module of $R$-Lie algebras} is given by $(M,P,\partial,\psi)$ where $M$ and $P$ are $R$-Lie algebras, $\partial\colon M\to P$ is a morphism between them, and \linebreak $\psi\colon P\times M\to M$ is an action such that the diagram
\begin{align*}
\xymatrix{
M\times M \ar[r]^-{\chi_M} \ar[d]_-{\partial\times 1_M} & M \ar@{=}[d]\\
P\times M \ar[r]^-{\psi} \ar[d]_-{1_P\times \partial} & M \ar[d]^-{\partial}\\
P\times P \ar[r]_-{\chi_P} & P\\
}
\end{align*}
commutes. That is, such that $[m,m']={}^{\partial(m)}m'$ and $\partial({}^{p}m)=[p,\partial(m)]$.
\end{defi}

Again by using the equivalence between the actions and the internal actions we can find the equivalent definition of internal crossed modules, first appeared in~\cite{Jan03}: this is actually a simplified version due to the fact that in $\LieAlg$ the \lq\lq Smith-is-Huq\rq\rq\ condition holds (see~\cite{MFVdL12} for further details).

\begin{defi}
\label{defi:definition of internal crossed module of LieAlg}
An \emph{internal crossed module of $R$-Lie algebras} is given by $(M,P,\partial,\xi)$ where $M$ and $P$ are $R$-Lie algebras, $\partial\colon M\to P$ is a morphism between them, and $\xi\colon P\flat M\to M$ is an internal action such that the following diagram commutes
\begin{align*}
\xymatrix{
M\flat M \ar[r]^-{\chi_M} \ar[d]_-{\partial\flat 1_M} & M \ar@{=}[d]\\
P\flat M \ar[r]^-{\xi} \ar[d]_-{1_P\flat \partial} & M \ar[d]^-{\partial}\\
P\flat P \ar[r]_-{\chi_P} & P\\
}
\end{align*}
\end{defi}

\begin{prop}
\label{prop:two coterminal crossed modules induce compatible actions in LieAlg}
Let $M$ and $N$ be $R$-Lie algebras. Consider two crossed module structures $(M\xrightarrow{\mu}P,\psi_M)$ and $(N\xrightarrow{\nu}P,\psi_N)$
\[
\xymatrix{
& M\ar[d]^-{\mu}\\
N \ar[r]_-{\nu} & P
}
\]
and construct two induced actions $\psi^M_N$ and $\psi^N_M$ as follows:
\begin{align*}
\xymatrix{
M\times N \ar[rd]_-{\mu\times 1_N} \ar[rr]^-{\psi^M_N} && N\\
& P\times N \ar[ru]_-{\psi_M}
}
&&
\xymatrix{
N\times M \ar[rd]_-{\nu\times 1_M} \ar[rr]^-{\psi^N_M} && M\\
& P\times M \ar[ru]_-{\psi_N}
}
\end{align*}
These two actions are compatible.
\end{prop}
\begin{proof}
We need to prove the equation ${}^{({}^{n}m)}n'=[n',{}^{m}n]$ by using the crossed module conditions $[m,m']={}^{\mu(m)}m'$ and $\mu({}^{p}m)=[p,\mu(m)]$, and $[n,n']={}^{\nu(n)}n'$ and $\nu({}^{p}n)=[p,\nu(n)]$. We have the chain of equalities
\begin{align*}
{}^{({}^{n}m)}n'&={}^{({}^{\nu(n)}m)}n'={}^{\mu({}^{\nu(n)}m)}n'={}^{[\nu(n),\mu(m)]}n'\\
&={}^{-[\mu(m),\nu(n)]}n'={}^{-\nu({}^{\mu(m)}n)}n'={}^{\nu(-{}^{\mu(m)}n)}n'\\
&=[-{}^{\mu(m)}n,n']=[n',{}^{\mu(m)}n]=[n',{}^{m}n].
\end{align*}
For the second equation, the reasoning is the same.
\end{proof}

Imitating what has been done in the case of groups (\cite{Whi41,GH87}), we are able to define the Peiffer product of two Lie algebras acting on each other (this was firstly defined in~\cite{Khm99}).
\begin{defi}
\label{defi:Peiffer product of Lie algebras}
Given two Lie algebras $M$ and $N$ acting on each other, consider their coproduct $M+N$ and its ideal $K$, generated by the elements 
\[
\left({}^{n}m\right)-\left[n,m\right]
\qquad\text{and}\qquad
\left({}^{m}n\right)-\left[m,n\right],
\]
for $m\in M$ and $n\in N$. We define the \emph{Peiffer product $M\bowtie N$ of $M$ and $N$} as the quotient
\begin{align*}
\xymatrix{
K \ar@{>->}[r] & M+N \ar@{>>}[r]^-{q_K} & \frac{M+N}{K} =:M\bowtie N.
}
\end{align*}
\end{defi}

\begin{rmk}
Notice that an equivalent way of defining the Peiffer product is the following coequalizer
\begin{align*}
\xymatrixcolsep{4pc}
\xymatrix{
(N\flat M)+(M\flat N) \ar@<2pt>[r]^-{\binom{k_{N,M}}{k_{M,N}}} \ar@<-2pt>[r]_-{\xi^N_M+\xi^M_N} & M+N \ar@{>>}[r]^-{q} & M\bowtie N
}
\end{align*}
In order to show that this definition is equivalent to the previous one, consider the morphism $q_K$ given by the first definition. It is easy to see that 
\begin{align*}
\begin{cases}
q_K\circ i_M\circ \xi^N_M = q_K\circ k_{N,M}\\
q_K\circ i_N\circ \xi^M_N = q_K\circ k_{M,N}
\end{cases}
\end{align*}
since this is exactly what taking the quotient by $K$ means. But this is the same as saying that
\begin{align*}
\begin{cases}
q_K\circ (\xi^N_M+\xi^M_N) \circ i_{N\flat M} = q_K\circ k_{N,M}\\
q_K\circ (\xi^N_M+\xi^M_N) \circ i_{M\flat N} = q_K\circ k_{M,N}
\end{cases}
\end{align*}
which in turn is 
\begin{align*}
q_K\circ (\xi^N_M+\xi^M_N)=q_K\circ \binom{k_{N,M}}{k_{M,N}}.
\end{align*}
The universal property of the coequalizer is given by the universal property of the quotient by $K$ in a straightforward way.
\end{rmk}

Since $K$ acts trivially on both $M$ and $N$ we can define induced actions $\xi^{M\bowtie N}_M$ and $\xi^{M\bowtie N}_N$ of $M \bowtie N$ on $M$ and $N$, that is such that the following diagrams commute
\begin{align}
\label{diag:diagram involving action of coproduct and action of Peiffer product in LieAlg}
\vcenter{\xymatrix{
(M+N)\flat M \ar[rd]_-{\xi^{M+N}_M} \ar[r]^-{q\flat 1_M} & (M\bowtie N)\flat M \ar[d]^-{\xi^{M\bowtie N}_M}\\
& M
}}
&&
\vcenter{\xymatrix{
(M+N)\flat N \ar[rd]_-{\xi^{M+N}_N} \ar[r]^-{q\flat 1_N} & (M\bowtie N)\flat N \ar[d]^-{\xi^{M\bowtie N}_N}\\
& M
}}
\end{align}
We can describe these actions of the Peiffer product through its universal property, but in order to do this, we need a preliminary lemma and a remark.

\begin{lemma}
\label{lemma:-bX preserves coequalizers of reflexive graphs}
Let $X$ be an object in a semi-abelian category $\A$. Then the functor $-\flat X\colon\A\to\A$ preserves coequalizers of reflexive graphs.
\end{lemma}
A proof of this result, based on a proposition in~\cite{HVdL11v1}, is straightforward but a bit involved, and can be found in the paper in preparation \cite{dMVdL19}.

\begin{rmk}
\label{rmk:the maps coequalized by the Peiffer product form a reflexive graph in LieAlg}
Notice that the two compositions
\begin{align*}
\xymatrixcolsep{6pc}
\xymatrix{
M+N \ar[r]^-{\eta^N_M+\eta^M_N} & (N\flat M)+(M\flat N) \ar@<2pt>[r]^-{\binom{k_{N,M}}{k_{M,N}}} \ar@<-2pt>[r]_-{\xi^N_M+\xi^M_N} & M+N
}
\end{align*}
are given by $1_{M+N}$. Hence we have that
\begin{align*}
\xymatrixcolsep{8pc}
\xymatrix{
(N\flat M)+(M\flat N) \ar@/^1.5pc/[r]^-{\binom{k_{N,M}}{k_{M,N}}} \ar@/_1.5pc/[r]_-{\xi^N_M+\xi^M_N} & M+N \ar[l]|-{\eta^N_M+\eta^M_N}
}
\end{align*}
is a reflexive graph.
\end{rmk}

Lemma~\ref{lemma:-bX preserves coequalizers of reflexive graphs} implies that $q\flat 1_M$ is the coequalizer of $\binom{k_{N,M}}{k_{M,N}}\flat 1_M$ and $(\xi^N_M+\xi^M_N)\flat 1_M$ and that $q\flat 1_N$ is the coequalizer of $\binom{k_{N,M}}{k_{M,N}}\flat 1_N$ and $(\xi^N_M+\xi^M_N)\flat 1_N$. We want to use these universal properties to define induced actions $\xi^{M\bowtie N}_M$ and $\xi^{M\bowtie N}_N$ of $M \bowtie N$ on $M$ and $N$ as in the next two diagrams
\begin{align*}
\xymatrixcolsep{3.5pc}
\xymatrix{
((N\flat M)+(M\flat N))\flat M \ar@<2pt>[r]^-{\binom{k_{N,M}}{k_{M,N}}\flat 1_M} \ar@<-2pt>[r]_-{\left(\xi^N_M+\xi^M_N\right)\flat 1_M} & (M+N)\flat M \ar[rd]_-{\xi^{M+N}_M} \ar[r]^-{q\flat 1_M} & (M\bowtie N)\flat M \ar@{.>}[d]^-{\xi^{M\bowtie N}_M}\\
& & M
}
\\
\xymatrixcolsep{3.5pc}
\xymatrix{
((N\flat M)+(M\flat N))\flat N \ar@<2pt>[r]^-{\binom{k_{N,M}}{k_{M,N}}\flat 1_N} \ar@<-2pt>[r]_-{\left(\xi^N_M+\xi^M_N\right)\flat 1_N} & (M+N)\flat N \ar[rd]_-{\xi^{M+N}_N} \ar[r]^-{q\flat 1_N} & (M\bowtie N)\flat N \ar@{.>}[d]^-{\xi^{M\bowtie N}_N}\\
& & N
}
\end{align*}
In order to do so, we need the following result.
\begin{prop}
\label{prop:the coproduct action coequalizes the two maps in LieAlg}
The action $\xi^{M+N}_M$ coequalizes $\binom{k_{N,M}}{k_{M,N}}\flat 1_M$ and $(\xi^N_M+\xi^M_N)\flat 1_M$. Similarly, the action $\xi^{M+N}_N$ coequalizes $\binom{k_{N,M}}{k_{M,N}}\flat 1_N$ and $(\xi^N_M+\xi^M_N)\flat 1_N$. 
\end{prop}
\begin{proof}
We need to show that
\begin{equation}
\label{eq:coproduct lie action coequalizes maps}
\xi^{M+N}_M\left(\left(\binom{k_{N,M}}{k_{M,N}}\flat 1_M\right)(s)\right)=\xi^{M+N}_M\left(\left(\left(\xi^N_M+\xi^M_N\right)\flat 1_M\right)(s)\right)
\end{equation}
holds for each element $s\in((N\flat M)+(M\flat N))\flat M$. By Lemma~\ref{lemma:rewriting lemma} and Remark~\ref{rmk:choosing last element is possible}, it suffices to check this for the generators of the form $s=[x_k,[\ldots,[x_1,\overline{m}]\cdots]]$ with $x_i\in N\flat M$ or $x_i\in M\flat N$ and $\overline{m}\in M$.
This means that to prove (\ref{eq:coproduct lie action coequalizes maps}) it suffices to show the equality
\begin{equation}
\label{eq:the internal coproduct action coequalizes things in LieAlg}
\xi^{M+N}_M\left([x_k,[\ldots,[x_1,\overline{m}]\cdots]]\right)=\xi^{M+N}_M\left([\epsilon(x_k),[\ldots,[\epsilon(x_1),\overline{m}]\cdots]]\right)
\end{equation}
where 
\[
\epsilon(x_i)=
\begin{cases}
\xi^M_N(x_i) & \text{if $x_i\in M\flat N$,}\\
\xi^N_M(x_i) & \text{if $x_i\in N\flat M$.}
\end{cases}
\]
In order to see this, we can use the decomposability of the action $\xi^{M+N}_M$ on both sides of (\ref{eq:the internal coproduct action coequalizes things in LieAlg}) obtaining that the one on the left becomes
\[
\xi^{M+N}_M\left([x_k,\xi^{M+N}_M\left([\ldots,\xi^{M+N}_M\left([x_1,\overline{m}]\right)\cdots]\right)]\right)
\]
whereas the one on the right becomes 
\[
\xi^{M+N}_M\left([\epsilon(x_k),\xi^{M+N}_M\left([\ldots,\xi^{M+N}_M\left([\epsilon(x_1),\overline{m}]\right)\cdots]\right)]\right).
\]
This means that it suffices to show
\[
\xi^{M+N}_M\left([x,\overline{m}]\right)=\xi^{M+N}_M\left([\epsilon(x),\overline{m}]\right)
\]
for $x\in M\flat N$ or $x\in N\flat M$, but this is given again by decomposability of $\xi^{M+N}_M$.

Finally, we repeat the whole reasoning with $\xi^{M+N}_N$.
\end{proof}

\begin{prop}
\label{prop:Peiffer has two crossed module structures in LieAlg}
We have two crossed module structures
\begin{align*}
(M\xrightarrow{l_M}M\bowtie N,\xi^{M\bowtie N}_M) && (N\xrightarrow{l_N}M\bowtie N,\xi^{M\bowtie N}_N)
\end{align*}
where the actions of the Peiffer product are induced as above and the morphisms $l_M$ and $l_N$ are defined through
\begin{align}
\label{diag:defi of inclusions in Peiffer product in LieAlg}
\vcenter{\xymatrix{
M \ar@{^{(}->}[rd]^-{i_M} \ar@/_/[rdd]_-{l_M} && N \ar@{_{(}->}[ld]_-{i_N} \ar@/^/[ldd]^-{l_N}\\
 & M+N \ar@{>>}[d]^-{q}\\
 & M\bowtie N
}}
\end{align}
\end{prop}
\begin{proof}
We will prove the claim only for $\xi^{M\bowtie N}_M$, since the proof in the other case uses the same strategy. We need to show the commutativity of the following squares
\begin{align*}
\xymatrixcolsep{3pc}
\xymatrix{
M\flat M \ar[r]^-{\chi_M} \ar[d]_-{l_M\flat 1_M} & M \ar@{=}[d]\\
(M\bowtie N)\flat M \ar[r]^-{\xi^{M\bowtie N}_M} \ar[d]_-{1_{M\bowtie N}\flat l_M} & M \ar[d]^-{l_M}\\
(M\bowtie N)\flat (M\bowtie N) \ar[r]_-{\chi_{M\bowtie N}} & (M\bowtie N)
}
\end{align*}
For what concerns the commutativity of the upper square, we have the chain of equalities
\begin{align*}
\xi^{M\bowtie N}_M \circ (l_M\flat 1_M)&=\xi^{M\bowtie N}_M \circ (q_K\flat 1_M)\circ (i_M\flat 1_M)\\
&=\xi^{M+N}_M \circ (i_M\flat 1_M)\\
&=\chi_M
\end{align*}
given by the definition of the coproduct action and of the Peiffer product action.

As for the lower square, we can precompose with the regular epimorphism $q\flat 1_M$: this shows that the required commutativity is equivalent to the equation 
\[
q\circ \chi_{M+N}\circ (1\flat i_M)=q\circ i_M\circ \xi^{M+N}_M. 
\]
Consider a generator $[s_k,[\ldots,[s_1,m]\cdots]]\in (M+N)\flat M$ with $m\in M$ and \linebreak $s_j\in M+N$ or $s_j\in M$ (see Lemma~\ref{lemma:rewriting lemma} and Remark~\ref{rmk:choosing last element is possible}): we want to show that
\begin{equation}
\label{eq:equivalent condition to precrossed module condition in Liealg}
q\left(\xi^{M+N}_M\left([s_k,[\ldots,[s_1,m]\cdots]]\right)\right)=q\left([s_k,[\ldots,[s_1,m]\cdots]]\right).
\end{equation}
We are going to prove this by induction on $k$: 
\begin{itemize}
\item If $k=0$ we trivially have
\[
q\left(\xi^{M+N}_M\left(m\right)\right)=q(m);
\]
\item Suppose that (\ref{eq:equivalent condition to precrossed module condition in Liealg}) holds for $j<k$. Then by using the decomposability of $\xi^{M+N}_M$ and the equality
\[
q([s,m])=q(k_{N,M}([s,m]))=q\left(\xi^N_M([s,m])\right)
\]
induced from the definition of the Peiffer product as coequalizer, we have the chain of equalities
\begin{align*}
q\left(\xi^{M+N}_M\left([s_k,[\ldots,[s_1,m]\cdots]]\right)\right)&=q\left(\xi^{M+N}_M\left(\left[s_k,\left[\ldots,\xi^{M+N}_M\left([s_1,m]\right)\cdots\right]\right]\right)\right)\\
&=q\left(\left[s_k,\left[\ldots,\xi^{M+N}_M\left([s_1,m]\right)\cdots\right]\right]\right)\\
&=q\left(\left[s_k,\left[\ldots,\xi^N_M\left([s_1,m]\right)\cdots\right]\right]\right)\\
&=\left[q\left(s_k\right),\left[\ldots,q\left(\xi^N_M\left([s_1,m]\right)\right)\cdots\right]\right]\\
&=\left[q\left(s_k\right),\left[\ldots,q\left([s_1,m]\right)\cdots\right]\right]\\
&=q\left(\left[s_k,\left[\ldots,[s_1,m]\cdots\right]\right]\right).
\end{align*}
Notice that the induction hypothesis is used for the equality on the second line, considering $\xi^{M+N}_M\left([s_1,m]\right)$ as $m'\in M$.\qedhere
\end{itemize}
\end{proof}

Furthermore we know that the actions $\xi^M_N$ and $\xi^N_M$ are in turn induced by $\xi^{M\bowtie N}_M$ and $\xi^{M\bowtie N}_N$ through the morphisms $l_M$ and $l_N$, that is
\begin{align*}
\xymatrix{
M\flat N \ar[r]^-{l_M\flat 1_N} \ar[rd]_-{\xi^M_N} & (M\bowtie N)\flat N \ar[d]^-{\xi^{M\bowtie N}_N}\\
& N
}
&&
\xymatrix{
N\flat M \ar[r]^-{l_N\flat 1_M} \ar[rd]_-{\xi^N_M} & (M\bowtie N)\flat M \ar[d]^-{\xi^{M\bowtie N}_M}\\
& M
}
\end{align*}
commute. This can be proved by using the definition of the coproduct actions and the commutativity of diagrams (\ref{diag:diagram involving action of coproduct and action of Peiffer product in LieAlg}) and (\ref{diag:defi of inclusions in Peiffer product in LieAlg}).

Putting together Proposition~\ref{prop:two coterminal crossed modules induce compatible actions in LieAlg} and Proposition~\ref{prop:Peiffer has two crossed module structures in LieAlg}, we find the following characterization of compatible actions.

\begin{thm}
\label{thm:remark 2.16 in BL in LieAlg}
Consider two Lie algebras $M$ and $N$ acting on each other. These actions are compatible if and only if there exists a Lie algebra $L$ with two crossed module structures $(M\xrightarrow{\mu}L,\psi_M)$ and $(N\xrightarrow{\nu}L,\psi_N)$ such that the action of $M$ on~$N$ and the action of $N$ on $M$ are induced from $L$ and its actions, through $\mu$ and $\nu$.
\end{thm}

\section{The Peiffer product as a coproduct}

As a final result we want to show that the coproduct in $\XMod_L(\LieAlg)$ can be obtained through the Peiffer product: this coproduct has already been characterized in a different way in~\cite{CL00} by using semi-direct products instead of the Peiffer product, but this approach generalizes the one used for $\XMod_L(\Grp)$ in~\cite{Bro84}. Consequently, we also obtain that the Peiffer product defined above (and hence the one from~\cite{Khm99}) coincides with the one defined in~\cite{CMM17} when restricted to $\LieAlg$. 

\begin{defi}
\label{defi:definition of the action of L on the coproduct in LieAlg}
Given a pair of actions of $L$ respectively on $M$ and on $N$, we can define an action of $L$ on the coproduct $M+N$ by imposing the equalities
\[
{}^{l}s:=
\begin{cases}
{}^{l}m & \text{if $s=m\in M$}\\
{}^{l}n & \text{if $s=n\in N$}\\
\left[{}^{l}s_1,s_2\right]+\left[s_1,{}^{l}s_2\right], & \text{if $s=\left[s_1,s_2\right]\in M+N$}
\end{cases}
\]
and by extending the definition by linearity. In order to see that this is well defined it suffices to use Lemma~\ref{lemma:rewriting lemma} and induction on the length of $s\in M+N$.
\end{defi}

\begin{prop}
\label{prop:the action restricts to K in LieAlg}
The action $\psi_{M+N}$ restricts to an action on $K$. Consequently it induces an action $\psi_{M\bowtie N}$ of $L$ on the quotient $M\bowtie N$. 
\end{prop}
\begin{proof}
Let's show that ${}^{l}k$ lies in $K$ (that is $q\left({}^{l}k\right)=0$) as soon as $k\in K$. In order to do this, it suffices to prove it for the generators 
\[
\left({}^{n}m\right)-\left[n,m\right]
\qquad\text{and}\qquad
\left({}^{m}n\right)-\left[m,n\right],
\]
We prove it for the first one since the reasoning can be repeated for the other one:
\begin{align*}
q\left({}^{l}\left({}^{n}m-[n,m]\right)\right)&=q\left({}^{l}\left({}^{n}m\right)\right)-q\left({}^{l}\left([n,m]\right)\right)\\
&=q\left({}^{[l,\nu(n)]}m\right)+q\left({}^{\nu(n)}\left({}^{l}m\right)\right)-q\left([{}^{l}n,m]\right)-q\left([n,{}^{l}m]\right)\\
&=q\left({}^{\nu\left({}^{l}n\right)}m\right)+q\left(\left[n,{}^{l}m\right]\right)-q\left([{}^{l}n,m]\right)-q\left([n,{}^{l}m]\right)\\
&=q\left({}^{\left({}^{l}n\right)}m\right)-q\left([{}^{l}n,m]\right)\\
&=q\left({}^{\left({}^{l}n\right)}m-[{}^{l}n,m]\right)=0
\end{align*}
For the second part of the claim it suffices to apply Theorem 5.5 in~\cite{Met17} and use the fact that, as shown in~\cite{MM10.2}, $\LieAlg$ is a strongly protomodular category in the sense of~\cite{Bou00}.
\end{proof}

\begin{prop}
\label{prop:the action on the Peiffer product is part of a crossed module}
If in the previous situation the actions on $M$ and $N$ are part of crossed module structures $(M\xrightarrow{\mu}L,\psi_M)$ and $(N\xrightarrow{\nu}L,\psi_N)$, then also the induced action on the Peiffer product is part of a crossed module structure
\[
(M\bowtie N\xrightarrow{\verticalbinomial{\mu}{\nu}}L,\psi_{M\bowtie N}).
\]
\end{prop}
\begin{proof}
Since $q\colon M+N\to M\bowtie N$ is an epimorphism, it suffices to show that for each $s,s'\in M+N$ and for each $l\in L$ the equalities
\begin{align*}
q\left(\binom{\mu}{\nu}\left({}^{l}s\right)\right)=q\left(\left[l,\binom{\mu}{\nu}(s)\right]\right) && q\left({}^{\binom{\mu}{\nu}(s')}(s)\right)=q\left([s',s]\right)
\end{align*}
hold. 

We are going to show them only in the case in which $s=\left[m,n\right]$ and $s'=\left[m',n'\right]$, but the reasoning easily generalizes to give the induction step needed for a complete proof by induction on the complexity of $s$ and $s'$.
Notice that we already have the equalities
\begin{align*}
\binom{\mu}{\nu}\left({}^{l}\left[m,n\right]\right)&=\binom{\mu}{\nu}\left(\left[{}^{l}m,n\right]+\left[m,{}^{l}n\right]\right)\\
&=\left[\mu\left({}^{l}m\right),\nu(n)\right]+\left[\mu(m),\nu\left({}^{l}n\right)\right]\\
&=\left[\left[l,\mu(m)\right],\nu(n)\right]+\left[\mu(m),\left[l,\nu(n)\right]\right]\\
&=\left[l,\left[\mu(m),\nu(n)\right]\right]\\
&=\left[l,\binom{\mu}{\nu}\left(\left[m,n\right]\right)\right],
\end{align*}
hence by applying $q$ to both sides we obtain the first equation. As for the second one we have
\begin{align*}
q\left({}^{\binom{\mu}{\nu}([m',n'])}([m,n])\right)&=q\left({}^{\mu(m')}\left({}^{\nu(n')}[m,n]\right)-{}^{\nu(n')}\left({}^{\mu(m')}[m,n]\right)\right)\\
&=q\left(\left[m,\left[{}^{m'}n',n\right]\right]-\left[n,\left[m,{}^{n'}m'\right]\right]\right)\\
&=\left[q\left(m\right),\left[q\left({}^{m'}n'\right),q\left(n\right)\right]\right]-\left[q\left(n\right),\left[q\left(m\right),q\left({}^{n'}m'\right)\right]\right]\\
&=\left[q\left(m\right),\left[q([m',n']),q\left(n\right)\right]\right]-\left[q\left(n\right),\left[q\left(m\right),q([n',m'])\right]\right]\\
&=q\left([[m',n'],[m,n]]\right).\qedhere
\end{align*}
\end{proof}

\begin{prop}
\label{prop:the Peiffer crossed module is the coproduct in LieAlg}
Given a pair of $L$-crossed modules
\[
(M\xrightarrow{\mu}L,\psi_M)
\qquad\text{and}\qquad
(N\xrightarrow{\nu}L,\psi_N),
\]
their coproduct in $\XMod_L(\LieAlg)$ is given by $(M\bowtie N\xrightarrow{\verticalbinomial{\mu}{\nu}}L,\psi_{M\bowtie N})$.
\end{prop}
\begin{proof}
Suppose we have a crossed module $(Z\xrightarrow{z}L,\psi_Z)$ with two morphisms $(z_M,1_L)$ and $(z_N,1_L)$ as in the following diagram
\[
\xymatrixrowsep{3pc}
\xymatrix{
(M\xrightarrow{\mu}L,\psi_M) \ar@/_/[rdd]_-{(z_M,1_L)} \ar[rd]^-{(l_M,1_L)} && (N\xrightarrow{\nu}L,\psi_N) \ar@/^/[ldd]^-{(z_N,1_L)} \ar[ld]_-{(l_N,1_L)}\\
& (M\bowtie N\xrightarrow{\verticalbinomial{\mu}{\nu}}L,\psi_{M\bowtie N}) \ar@{.>}[d]|-{(\verticalbinomial{z_M}{z_N},1_L)}\\
& (Z\xrightarrow{z}L,\psi_Z)
} 
\]
We want to construct the dotted morphism of crossed modules such that the two triangles commute. The first step is constructing the arrow $\verticalbinomial{z_M}{z_N}$ through the diagram 
\begin{align*}
\xymatrixcolsep{4pc}
\vcenter{\xymatrix{
(N\flat M)+(M\flat N) \ar@<2pt>[r]^-{\binom{k_{N,M}}{k_{M,N}}} \ar@<-2pt>[r]_-{\xi^N_M+\xi^M_N} & M+N \ar@{>>}[r]^-{q} \ar[rd]_-{\binom{z_M}{z_N}} & M\bowtie N \ar@{.>}[d]^-{\verticalbinomial{z_M}{z_N}}\\
&& Z
}}
\end{align*}
In order to do so we need to show that $\binom{z_M}{z_N}$ coequalizes the arrows on the left. This is done by using the Peiffer condition for $(Z\xrightarrow{z}L,\psi_Z)$ and the fact that $(z_M,1_L)$ and $(z_N,1_L)$ are morphisms of crossed modules:
\begin{align*}
\binom{z_M}{z_N}\circ(\xi^N_M+\xi^M_N)&=\binom{z_M\circ \xi^N_M}{z_N\circ \xi^M_N}=\binom{\psi_Z\circ(\nu\flat z_M)}{\psi_Z\circ(\mu\flat z_N)}\\
&=\psi_Z\circ\binom{\nu\flat z_M}{\mu\flat z_N}=\psi_Z\circ\binom{(z\flat 1)\circ(z_N\flat z_M)}{(z\flat 1)\circ(z_M\flat z_N)}\\
&=\psi_Z\circ(z\flat 1)\circ\binom{z_N\flat z_M}{z_M\flat z_N}=\chi_Z\circ\binom{z_N\flat z_M}{z_M\flat z_N}\\
&=\binom{\binom{z_M}{z_N}\circ k_{N,M}}{\binom{z_M}{z_N}\circ k_{M,N}}=\binom{z_M}{z_N}\circ\binom{k_{N,M}}{k_{M,N}}.
\end{align*}
Finally we need to show the commutativity of the diagrams 
\begin{align*}
\xymatrix{
L\flat (M\bowtie N) \ar[d]_-{\psi_{M\bowtie N}} \ar[r]^-{1\flat \verticalbinomial{z_M}{z_N}} & L\flat Z \ar[d]^-{\psi_Z}\\
M\bowtie N \ar[r]_-{\verticalbinomial{z_M}{z_N}} & Z
}
&&
\xymatrix{
M\bowtie N \ar[d]_-{\verticalbinomial{\mu}{\nu}} \ar[r]^-{\verticalbinomial{z_M}{z_N}} & Z \ar[d]^-{z}\\
L \ar@{=}[r] & L
}
\end{align*}
To obtain the second one it suffices to precompose with the epimorphism $q$
\[
\verticalbinomial{\mu}{\nu}\circ q=\binom{\mu}{\nu}=z\circ\binom{z_M}{z_N}=\verticalbinomial{z_M}{z_N}\circ q
\]
whereas for the first one, we need to use the fact that $\LieAlg$ is an algebraically coherent category, and hence $1\flat l_M$ and $1\flat l_M$ are jointly strongly epimorphic, since $l_M$ and $l_N$ are so (see Theorem 3.18 in~\cite{CGVdL15b} for further details). This means that in order to prove the claim, we only need to check the commutativity of the outer rectangles
\begin{align*}
\xymatrix{
L\flat M \ar[r]^-{1\flat l_M} \ar[d]_-{\psi_M} & L\flat (M\bowtie N) \ar[d]_-{\psi_{M\bowtie N}} \ar[r]^-{1\flat \verticalbinomial{z_M}{z_N}} & L\flat Z \ar[d]^-{\psi_Z}\\
M \ar[r]_-{l_M} & M\bowtie N \ar[r]_-{\verticalbinomial{z_M}{z_N}} & Z
}
&&
\xymatrix{
L\flat N \ar[r]^-{1\flat l_N} \ar[d]_-{\psi_N} & L\flat (M\bowtie N) \ar[d]_-{\psi_{M\bowtie N}} \ar[r]^-{1\flat \verticalbinomial{z_M}{z_N}} & L\flat Z \ar[d]^-{\psi_Z}\\
N \ar[r]_-{l_N} & M\bowtie N \ar[r]_-{\verticalbinomial{z_M}{z_N}} & Z
}
\end{align*}
which is given by hypothesis.
\end{proof}

\section*{Acknowledgements}
I would like to thank Sandra Mantovani, Andrea Montoli and Tim Van der Linden for useful comments and suggestions.

\providecommand{\bysame}{\leavevmode\hbox to3em{\hrulefill}\thinspace}
\providecommand{\MR}{\relax\ifhmode\unskip\space\fi MR }
\providecommand{\MRhref}[2]{%
  \href{http://www.ams.org/mathscinet-getitem?mr=#1}{#2}
}
\providecommand{\href}[2]{#2}

\end{document}